\documentclass[reqno,tbtags,a4paper,12pt]{amsart}

\usepackage[in]{fullpage}
\usepackage[matrix,arrow,curve]{xy}
\usepackage{tikz}
\usetikzlibrary{snakes}

\usepackage{amsmath,amssymb,amsthm,amsfonts,amscd}
\usepackage{graphics,graphicx}

\usepackage{mathrsfs}
\definecolor{mygreen}{rgb}{0, 0.4, 0}
\definecolor{myblue}{rgb}{0, 0, 0.5}

\newcommand{\Out}{\mathop{\mathrm{Out}}}
\newcommand{\cd}{\mathop{\mathrm{cd}}\nolimits}
\newcommand{\lk}{\mathop{\mathrm{lk}}\nolimits}
\newcommand{\CV}{\mathbf{W}}
\newcommand{\Mod}{\mathop{\mathrm{Mod}}\nolimits}
\renewcommand{\epsilon}{\varepsilon}

\newcommand{\Homeo}{\mathrm{Homeo}}
\newcommand{\Sp}{\mathrm{Sp}}
\newcommand{\SL}{\mathrm{SL}}
\newcommand{\cur}{\mathcal{C}}
\newcommand{\W}{\mathcal{W}}
\newcommand{\CH}{\mathcal H}
\newcommand{\Stab}{\mathop{\mathrm{Stab}}\nolimits}
\newcommand{\T}{\mathcal{T}}
\newcommand{\bT}{\mathbf{T}}
\newcommand{\I}{\mathcal{I}}
\newcommand{\B}{\mathcal{B}}
\newcommand{\M}{\mathcal{M}}
\newcommand{\K}{\mathcal{K}}
\newcommand{\Q}{\mathbb{Q}}
\newcommand{\Z}{\mathbb{Z}}
\newcommand{\R}{\mathbb{R}}
\newcommand{\iA}{\mathcal{A}}
\newcommand{\rank}{\mathop{\mathrm{rank}}}

\newtheorem{theorem}{Theorem}[section] 
\newtheorem{propos}[theorem] {Proposition}
\newtheorem{cor}[theorem] {Corollary}
\newtheorem{lem}[theorem]{Lemma}
\newtheorem{quest}[theorem]{Question}
\theoremstyle{definition}
\newtheorem{remark}[theorem]{Remark}

\numberwithin{equation}{section}

\author{Alexander A. Gaifullin}
\address{Steklov Mathematical Institute of Russian Academy of Sciences, Moscow, Russia}
\address{Skolkovo Institute of Science and Technology, Skolkovo, Russia}
\address{Lomonosov Moscow State University, Moscow, Russia}
\address{Institute for Information Transmission Problems (Kharkevich Institute), Moscow, Russia}
\email{agaif@mi-ras.ru}

\thanks{}

\title[On the top homology group of Johnson kernel]{On the top homology group of Johnson kernel}
\date{}


\begin{document}

\maketitle

\begin{abstract}
The action of the mapping class group~$\Mod_g$ of an oriented surface~$\Sigma_g$ on the lower central series of~$\pi_1(\Sigma_g)$ defines the descending filtration in~$\Mod_g$ called the \textit{Johnson filtration}. The first two terms of it are the \textit{Torelli group}~$\I_g$ and the \textit{Johnson kernel}~$\K_g$. By a fundamental result of Johnson (1985), $\K_g$ is the subgroup of~$\Mod_g$ generated by all Dehn twists about separating curves. In 2007, Bestvina, Bux, and Margalit showed the group $\K_g$ has cohomological dimension~$2g-3$. We prove that the top homology group~$H_{2g-3}(\K_g)$ is not finitely generated. In fact, we show that it contains a free abelian subgroup of infinite rank, hence, the vector space~$H_{2g-3}(\K_g,\Q)$ is infinite-dimensional. Moreover, we prove that $H_{2g-3}(\K_g,\Q)$ is not finitely generated as a module over the group ring~$\Q[\I_g]$.
\end{abstract}

\section{Introduction}

Let $\Sigma_g$ be an oriented closed surface of genus~$g$. Recall that the \textit{mapping class group} of~$\Sigma_g$ is the group
$$
\Mod_g=\pi_0\bigl(\Homeo^+(\Sigma_g)\bigr),
$$
where $\Homeo^+(\Sigma_g)$ is the group of orientation preserving homeomorphisms of~$\Sigma_g$ on itself. 

Let $\Gamma=\pi_1(\Sigma_g)$ and let 
$$
\Gamma=\Gamma_0\supseteq\Gamma_1\supseteq\Gamma_2\supseteq\cdots
$$
be the lower central series of~$\Gamma$, that is, the series of subgroups inductively defined by $\Gamma_0=\Gamma$ and $\Gamma_{k+1}=[\Gamma,\Gamma_k]$. The mapping class group~$\Mod_g$ acts by outer automorphisms of~$\Gamma$. All subgroups~$\Gamma_k$ are characteristic, i.\,e., invariant under all automorphisms of~$\Gamma$. The \textit{Johnson filtration} is, by definition the decreasing sequence of subgroups
\begin{gather*}
\Mod_g\supseteq\I_g(1)\supseteq\I_g(2)\supseteq\I_g(3)\supseteq\cdots,\\
\I_g(k)=\ker\bigl(\Mod_g\to\Out(\Gamma/\Gamma_k)\bigr).
\end{gather*} 

The most important of the groups~$\I_g(k)$ are the \textit{Torelli group}~$\I_g=\I_g(1)$ and the \textit{Johnson kernel}~$\K_g=\I_g(2)$. Since $\Gamma/\Gamma_1=H_1(\Sigma_g)$, we see that $\I_g$ is exactly the subgroup of~$\Mod_g$ consisting of all mapping classes that act trivially on the homology of~$\Sigma_g$. Equivalently, $\I_g$ is the kernel of the natural  surjective homomorphism $\Mod_g\twoheadrightarrow\Sp(2g,\Z)$. 

The group~$\K_g$ can be equivalently described as the kernel of the surjective \textit{Johnson homomorphism} (see~\cite{Joh80}) 
\begin{equation}\label{eq_Johnson}
\tau\colon\I_g\twoheadrightarrow U=\wedge^3H_1(\Sigma_g)/H_1(\Sigma_g).
\end{equation}
Here $H_1(\Sigma_g)$ is embedded into~$\wedge^3H_1(\Sigma_g)$ by $x\mapsto x\wedge\Omega$, where $\Omega\in\wedge^2H_1(\Sigma_g)$ is the inverse of the intersection form. By a fundamental result of Johnson~\cite{Joh85a} (cf.~\cite{Put09}), $\K_g$ is precisely the subgroup of $\Mod_g$ generated by all Dehn twists about separating simple closed curves.

It is a classical result due to Dehn that the group~$\I_1$ is trivial. The group~$\I_2=\K_2$ was shown to be not finitely generated by McCullough and Miller~\cite{MCM86}; Mess~\cite{Mes92} proved that it is an infinitely generated free group. Johnson~\cite{Joh83} showed that for $g\ge 3$, the group~$\I_g$ is finitely generated, and described explicitly a finite set of generators. The question of whether the group~$\K_g$ is finitely generated turned out to be more complicated. Only recently Ershov and Sue He~\cite{ErHe17} proved that  $\K_g$ is finitely generated, provided that $g\ge12$. Then this result was extended to any genus $g\ge 4$ by Church, Ershov, and Putman~\cite{CEP17}. It is still unknown whether $\K_3$ is finitely generated. Also it is unknown for all $g\ge 3$ whether~$\I_g$ and~$\K_g$ are finitely presented.

A natural problem is to study the homology of the groups~$\I_g$ and~$\K_g$, $g\ge 3$. To simplify notation, we omit the symbol~$\Z$ in notation for homology and cohomology with integral coefficients. Below is the list of some known up to now facts on the homology of~$\I_g$ and~$\K_g$, $g\ge 3$:
\begin{itemize}
\item The group $H_1(\I_g)$, i.\,e., the abelianization of~$\I_g$ was computed explicitly by Johnson~\cite{Joh85b}. 
\item The group $H_1(\K_g)$, i.\,e., the abelianization of~$\K_g$ was shown to be finitely generated for $g\ge 4$ by Dimca and Papadima~\cite{DiPa13}. The rational homology group~$H_1(\K_g,\Q)$ was computed explicitly for $g\ge 6$ by Morita, Sakasai, and Suzuki~\cite{MSS17} using the description due to Dimca, Hain, and Papadima~\cite{DHP14}. The torsion of~$H_1(\K_g)$ is still not described explicitly.
\item The total rational homology groups $H_*(\I_g,\Q)=\bigoplus_{k\ge 0}H_k(\I_g,\Q)$, where $g\ge 7$, and  $H_*(\K_g,\Q)=\bigoplus_{k\ge 0}H_k(\K_g,\Q)$, where $g\ge 2$, were shown to be infinite-dimensional by Akita~\cite{Aki01}.
\item The groups~$H_3(\I_3)$ (Johnson, Millson, cf.~\cite{Mes92}) and $H_4(\I_3)$ (Hain~\cite{Hai02}) were shown to be not finitely generated. 
\item Bestvina, Bux, and Margalit~\cite{BBM07} computed the cohomological dimensions $\cd(\I_g)=3g-5$, $\cd(\K_g)=2g-3$, and showed that the top homology group $H_{3g-5}(\I_g)$ is not finitely generated. (In fact, their proof implies that $H_{3g-5}(\I_g)$ contains a free abelian subgroup of infinite rank.)
\item The author~\cite{Gai18} proved that for any pair $g$ and~$k$ satisfying $g\ge 3$ and $2g-3\le k\le 3g-6$ the group~$H_k(\I_g)$ contains a free abelian subgroup of infinite rank, hence, is not finitely generated.
\item Kassabov and Putman~\cite{KaPu18} showed that the group $H_2(\I_g)$ is finitely generated as an $\Sp(2g,\Z)$-module for all $g\ge 3$.
\end{itemize}

The computation of the cohomological dimensions of~$\I_g$ and~$\K_g$ and the proof that the top homology group~$H_{3g-5}(\I_g)$ is not finitely generated were obtained by Bestvina, Bux, and Margalit~\cite{BBM07} by studying the action of~$\I_g$ and~$\K_g$ on a special contractible cell complex~$\B_g$, which was called  the \textit{complex of cycles}. However, they did not manage to prove that the top homology group of~$\K_g$ is not finitely generated, too. This was also explicitly asked as Question~5.10 in~\cite{Mar18}. The main result of the present paper is the following answer to this question.

\begin{theorem}\label{theorem_main}
Suppose that  $g\ge 3$. Then the group $H_{2g-3}(\K_g)$ contains a free abelian subgroup of infinite rank \textnormal{(}hence, is not finitely generated\textnormal{)}. Equivalently, the vector space $H_{2g-3}(\K_g,\Q)$ is infinite-dimensional.
\end{theorem}

Since $\K_g$ is a normal subgroup of~$\I_g$, the group~$\I_g$ acts on the homology of~$\K_g$. Since $\K_g$ acts trivially on homology of itself, this action reduces to the action of the group $U\cong \I_g/\K_g$ defined by~\eqref{eq_Johnson}. Theorem~\ref{theorem_main} can be strengthened in the following way. 

\begin{theorem}\label{theorem_main2}
Suppose that  $g\ge 3$. Then  $H_{2g-3}(\K_g,\Q)$ is not finitely generated as a~$\Q[\I_g]$-module \textnormal{(}equivalently, as a~$\Q[U]$-module\textnormal{)}.
\end{theorem}

It is standard to use additive notation for the group~$U$. Hence, denoting the element of the group ring~$\Q[U]$  corresponding to an element~$\theta\in U$ again by~$\theta$, we would arrive to a confusion. We shall conveniently denote the element of~$\Q[U]$   corresponding to an element~$\theta\in U$ by a symbol~$t^{\theta}$. Then $t^0=1$ is the unit of the ring~$\Q[U]$  and $t^{\theta_1+\theta_2}=t^{\theta_1}t^{\theta_2}$ for all $\theta_1,\theta_2\in U$. The same notation will be used for~$\Z[U]$.

Now, let us describe explicitly an infinite set of linearly independent homology classes in~$H_{2g-3}(\K_g)$. (We say that elements of an abelian group are \textit{linearly independent} if they satisfy no non-trivial linear relation with integral coefficients, that is, if they form a basis of a free abelian subgroup.)

Recall the definition of an \textit{abelian cycle}. Let $h_1,\ldots,h_k$ be pairwise commuting elements of a group~$G$. Consider the homomorphism
$\chi\colon\Z^k\to G$ that sends the generator of the $i$th factor~$\Z$ to~$h_i$ for every~$i$. We put~$\iA(h_1,\ldots,h_k)=\chi_*(\mu_k)$, where $\mu_k$ is the standard generator of~$H_k(\Z^k)\cong\Z$. Homology classes $\iA(h_1,\ldots,h_k)$  are called \textit{abelian cycles}. 

We denote by~$T_{\gamma}$ the left Dehn twist about a simple closed curve~$\gamma$.

Throughout the whole paper we suppose that $g\ge 3$ and identify~$\Sigma_g$ with the surface shown in Fig.~\ref{figure_delta_epsilon}. Consider the system of simple closed curves $\delta_1,\ldots,\delta_g,\epsilon_2,\ldots,\epsilon_{g-2}$ in~$\Sigma_g$ shown in this figure. These simple closed curves are separating and pairwise disjoint. Therefore, the Dehn twists $T_{\delta_1},\ldots,T_{\delta_g},T_{\epsilon_2},\ldots,T_{\epsilon_{g-2}}$ belong to~$\K_g$ and pairwise commute. We consider the abelian cycle
\begin{equation}\label{eq_A}
\iA_0=\iA\left(T_{\delta_1},\ldots,T_{\delta_g},T_{\epsilon_2},\ldots,T_{\epsilon_{g-2}}\right)\in H_{2g-3}(\K_g).
\end{equation}
Further we shall show that $\iA_0\ne 0$, see Proposition~\ref{propos_infinite}. Hence $T_{\delta_1},\ldots,T_{\delta_g},T_{\epsilon_2},\ldots,T_{\epsilon_{g-2}}$ generate a free abelian subgroup of rank~$2g-3$ in~$\K_g$. Notice that by a result of Vautaw~\cite{Vau02} the Torelli group~$\I_g$, hence, the Johnson kernel~$\K_g$ does not contain a free abelian subgroup of rank greater than~$2g-3$.

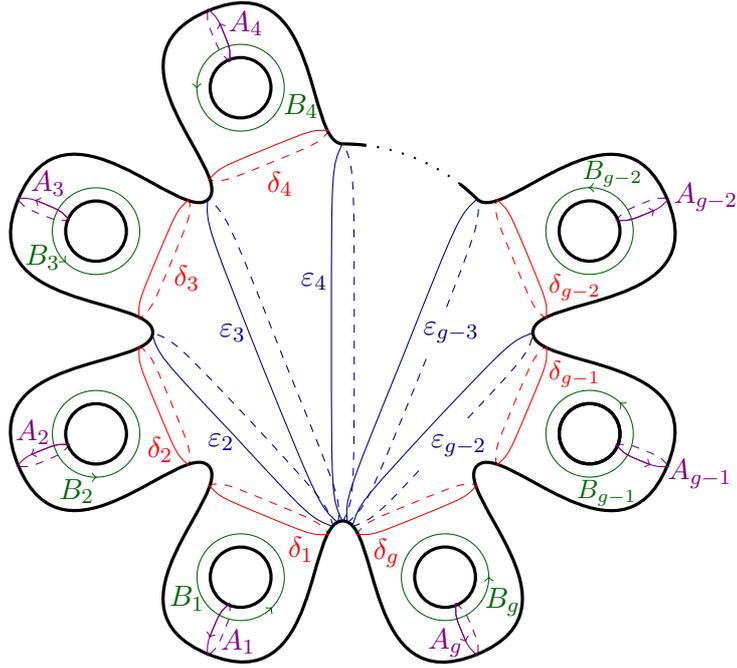
\begin{figure}

\begin{tikzpicture}[scale=2.5]
\small

\draw [myblue,dashed] (.05,-1.03) .. controls (.17,-.985) .. (.5645,-.554) .. controls (.969,-.123).. (1.005,0);

\path (.61,-.59) node [myblue,fill=white] {$\phantom{\epsilon_{g-2}}$};

\path (.61,-.6) node [myblue,fill=white] {$\epsilon_{g-2}$};

\draw [red,dashed] (0.071,-1.07) .. controls (0.128,-1.004) .. (0.374,-.9025) .. controls (0.62,-0.801) ..(0.707,-0.807);
\draw [red,dashed,rotate=45] (0.071,-1.07) .. controls (0.128,-1.004) .. (0.374,-.9025) .. controls (0.62,-0.801) ..(0.707,-0.807);
\draw [red,dashed,rotate=90] (0.071,-1.07) .. controls (0.128,-1.004) .. (0.374,-.9025) .. controls (0.62,-0.801) ..(0.707,-0.807);
\draw [red,dashed,rotate=-45] (0.071,-1.07) .. controls (0.128,-1.004) .. (0.374,-.9025) .. controls (0.62,-0.801) ..(0.707,-0.807);
\draw [red,dashed,rotate=-90] (0.071,-1.07) .. controls (0.128,-1.004) .. (0.374,-.9025) .. controls (0.62,-0.801) ..(0.707,-0.807);
\draw [red,dashed,rotate=-135] (0.071,-1.07) .. controls (0.128,-1.004) .. (0.374,-.9025) .. controls (0.62,-0.801) ..(0.707,-0.807);
\draw [red,dashed,rotate=180] (0.071,-1.07) .. controls (0.128,-1.004) .. (0.374,-.9025) .. controls (0.62,-0.801) ..(0.707,-0.807);

\draw [violet,dashed] (0.595,-1.436) .. controls (0.646,-1.484) .. (0.68,-1.5655) .. controls (0.714,-1.647) .. (0.711,-1.717) ;
\draw [violet,dashed,rotate=45] (0.595,-1.436) .. controls (0.646,-1.484) .. (0.68,-1.5655) .. controls (0.714,-1.647) .. (0.711,-1.717) ;
\draw [violet,dashed,rotate=90] (0.595,-1.436) .. controls (0.646,-1.484) .. (0.68,-1.5655) .. controls (0.714,-1.647) .. (0.711,-1.717) ;
\draw [violet,dashed,rotate=-45] (0.595,-1.436) .. controls (0.646,-1.484) .. (0.68,-1.5655) .. controls (0.714,-1.647) .. (0.711,-1.717) ;
\draw [violet,dashed,rotate=-90] (0.595,-1.436) .. controls (0.646,-1.484) .. (0.68,-1.5655) .. controls (0.714,-1.647) .. (0.711,-1.717) ;
\draw [violet,dashed,rotate=-135] (0.595,-1.436) .. controls (0.646,-1.484) .. (0.68,-1.5655) .. controls (0.714,-1.647) .. (0.711,-1.717) ;
\draw [violet,dashed,rotate=180] (0.595,-1.436) .. controls (0.646,-1.484) .. (0.68,-1.5655) .. controls (0.714,-1.647) .. (0.711,-1.717) ;

\draw [myblue,dashed] (0,-1.005) .. controls (.06,-0.88) .. (.06,0) .. controls (.06,0.88) .. (0,1.005);

\draw [myblue,dashed] (-.024,-1.01) .. controls (-.013, -.878).. (-.312,-.138) .. controls (-.611,0.622) ..  (-.71,0.71);

\draw [myblue,dashed] (.024,-1.01) .. controls (.123, -.922).. (.422,-.182) .. controls (.721,0.578) ..   (.71,0.71);

\draw [myblue,dashed] (-.05,-1.03) .. controls (-.086,-.907) .. (-.4805,-.476) .. controls (-.885,-.045).. (-1.005,0);

\draw [very thick] (0,-1) .. controls (.2,-1) and (0.107,-1.957) .. (0.707,-1.707) .. controls (1.307, -1.457) and (0.566,-0.848) .. (0.707,-0.707);
\draw  [very thick,rotate=45] (0,-1) .. controls (.2,-1) and (0.107,-1.957) .. (0.707,-1.707) .. controls (1.307, -1.457) and (0.566,-0.848) .. (0.707,-0.707);
\draw [very thick,rotate=90] (0,-1) .. controls (.2,-1) and (0.107,-1.957) .. (0.707,-1.707) .. controls (1.307, -1.457) and (0.566,-0.848) .. (0.707,-0.707);
\draw [very thick,rotate=-45] (0,-1) .. controls (.2,-1) and (0.107,-1.957) .. (0.707,-1.707) .. controls (1.307, -1.457) and (0.566,-0.848) .. (0.707,-0.707);
\draw [very thick,rotate=-90] (0,-1) .. controls (.2,-1) and (0.107,-1.957) .. (0.707,-1.707) .. controls (1.307, -1.457) and (0.566,-0.848) .. (0.707,-0.707);
\draw [very thick,rotate=-135] (0,-1) .. controls (.2,-1) and (0.107,-1.957) .. (0.707,-1.707) .. controls (1.307, -1.457) and (0.566,-0.848) .. (0.707,-0.707);
\draw [very thick,rotate=180] (0,-1) .. controls (.2,-1) and (0.107,-1.957) .. (0.707,-1.707) .. controls (1.307, -1.457) and (0.566,-0.848) .. (0.707,-0.707);
\draw [very thick] (0.538,-1.298) circle (.16);
\draw [very thick,rotate=45] (0.538,-1.298) circle (.16);
\draw [very thick,rotate=45] (0.538,-1.298) circle (.16);
\draw [very thick,rotate=90] (0.538,-1.298) circle (.16);
\draw [very thick,rotate=-45] (0.538,-1.298) circle (.16);
\draw [very thick,rotate=-90] (0.538,-1.298) circle (.16);
\draw [very thick,rotate=-135] (0.538,-1.298) circle (.16);
\draw [very thick,rotate=180] (0.538,-1.298) circle (.16);
\draw [very thick] (0,1) arc (90:83:1);
\draw [very thick] (0.707,.707) arc (45:52:1);
\draw [loosely dotted,thick] (0.113,.994) arc (83.5:52:1);

\path (-.22,-1.15) node [red] {$\delta_1$};
\path (-.96,-.62) node [red] {$\delta_2$};

\path (-.82,.29) node [red] {$\delta_{3}$};
\path (-.33,.79) node [red] {$\delta_{4}$};

\path (1.22,.24) node [red] {\footnotesize$\delta_{g-2}$};

\path (1.22,-.22) node [red] {\footnotesize$\delta_{g-1}$};

\path (.23,-1.17) node [red] {$\delta_g$};

\path (-.64,-.58) node [myblue] {$\epsilon_2$};
\path (-.58,0) node [myblue] {$\epsilon_3$};
\path (-.15,.27) node [myblue] {$\epsilon_4$};
\path (.57,0) node [myblue,fill=white] {$\epsilon_{g-3}$};

\path (-.55,-1.64) node [violet] {$A_1$};
\path (-.82,-1.41) node [mygreen] {$B_1$};
\path (.54,-1.645) node [violet] {$A_g$};
\path (.84,-1.43) node [mygreen] {$B_g$};
\path (-1.63,-.53) node [violet] {$A_2$};
\path (-1.4,-.84) node [mygreen] {$B_2$};
\path (-1.55,.79) node [violet] {$A_3$};
\path (-1.58,.4) node [mygreen] {$B_3$};
\path (-.51,1.64) node [violet] {$A_4$};
\path (-.22,1.2) node [mygreen] {$B_4$};

\path (1.88,-.75) node [violet] {$A_{g-1}$};
\path (1.39,-.85) node [mygreen] {\footnotesize$B_{g-1}$};

\path (1.91,.72) node [violet] {$A_{g-2}$};
\path (1.42,.84) node [mygreen] {\footnotesize$B_{g-2}$};

\draw [red] (0.071,-1.07) .. controls (0.158,-1.076) ..  (0.404,-.9745)  .. controls (0.65,-0.873) ..(0.707,-0.807);
\draw [red,rotate=45] (0.071,-1.07) .. controls (0.158,-1.076) .. (0.404,-.9745) .. controls (0.65,-0.873) ..(0.707,-0.807);
\draw [red,rotate=90] (0.071,-1.07) .. controls (0.158,-1.076) .. (0.404,-.9745) .. controls (0.65,-0.873) ..(0.707,-0.807);
\draw [red,rotate=-45] (0.071,-1.07) .. controls (0.158,-1.076) .. (0.404,-.9745) .. controls (0.65,-0.873) ..(0.707,-0.807);
\draw [red,rotate=-90] (0.071,-1.07) .. controls (0.158,-1.076) .. (0.404,-.9745) .. controls (0.65,-0.873) ..(0.707,-0.807);
\draw [red,rotate=-135] (0.071,-1.07) .. controls (0.158,-1.076) .. (0.404,-.9745) .. controls (0.65,-0.873) ..(0.707,-0.807);
\draw [red,rotate=180] (0.071,-1.07) .. controls (0.158,-1.076) .. (0.404,-.9745) .. controls (0.65,-0.873) ..(0.707,-0.807);

\draw [mygreen,->] (0.768,-1.298) arc (0:361:.23);
\draw [mygreen,rotate=45,->] (0.768,-1.298) arc (0:361:.23);
\draw [mygreen,rotate=90,->] (0.768,-1.298) arc (0:361:.23);
\draw [mygreen,rotate=-45,->] (0.768,-1.298) arc (0:361:.23);
\draw [mygreen,rotate=-90,->] (0.768,-1.298) arc (0:361:.23);
\draw [mygreen,rotate=-135,->] (0.768,-1.298) arc (0:361:.23);
\draw [mygreen,rotate=180,->] (0.768,-1.298) arc (0:361:.23);

\draw [violet,->] (0.595,-1.436) .. controls (0.592,-1.506) .. (.648,-1.6415);
\draw [violet] (0.595,-1.436) .. controls (0.592,-1.506) .. (.626,-1.5875) .. controls (0.66,-1.669) .. (0.711,-1.717) ;
\draw [violet,->,rotate=45] (0.595,-1.436) .. controls (0.592,-1.506) .. (.648,-1.6415);
\draw [violet,rotate=45] (0.595,-1.436) .. controls (0.592,-1.506) .. (.626,-1.5875) .. controls (0.66,-1.669) .. (0.711,-1.717) ;
\draw [violet,->,rotate=90] (0.595,-1.436) .. controls (0.592,-1.506) .. (.648,-1.6415);
\draw [violet,rotate=90] (0.595,-1.436) .. controls (0.592,-1.506) .. (.626,-1.5875) .. controls (0.66,-1.669) .. (0.711,-1.717) ;
\draw [violet,->,rotate=-45] (0.595,-1.436) .. controls (0.592,-1.506) .. (.648,-1.6415);
\draw [violet,rotate=-45] (0.595,-1.436) .. controls (0.592,-1.506) .. (.626,-1.5875) .. controls (0.66,-1.669) .. (0.711,-1.717) ;
\draw [violet,->,rotate=-90] (0.595,-1.436) .. controls (0.592,-1.506) .. (.648,-1.6415);
\draw [violet,rotate=-90] (0.595,-1.436) .. controls (0.592,-1.506) .. (.626,-1.5875) .. controls (0.66,-1.669) .. (0.711,-1.717) ;
\draw [violet,->,rotate=-135] (0.595,-1.436) .. controls (0.592,-1.506) .. (.648,-1.6415);
\draw [violet,rotate=-135] (0.595,-1.436) .. controls (0.592,-1.506) .. (.626,-1.5875) .. controls (0.66,-1.669) .. (0.711,-1.717) ;
\draw [violet,->,rotate=180] (0.595,-1.436) .. controls (0.592,-1.506) .. (.648,-1.6415);
\draw [violet,rotate=180] (0.595,-1.436) .. controls (0.592,-1.506) .. (.626,-1.5875) .. controls (0.66,-1.669) .. (0.711,-1.717) ;

\draw [myblue] (0,-1.005) .. controls (-.06,-0.88) .. (-.06,0) .. controls (-.06,0.88) .. (0,1.005);

\draw [myblue] (-.024,-1.01) .. controls (-.123, -.922).. (-.422,-.182) .. controls (-.721,0.578) ..   
(-.71,0.71);

\draw [myblue] (.024,-1.01) .. controls (.013, -.878).. (.312,-.138) .. controls (.611,0.622) ..  (.71,0.71);

\draw [myblue] (-.05,-1.03) .. controls (-.17,-.985) .. (-.5645,-.554) .. controls (-.969,-.123).. (-1.005,0);

\draw [myblue] (.05,-1.03) .. controls (.086,-.907) .. (.4805,-.476) .. controls (.885,-.045).. (1.005,0);

\end{tikzpicture}

\caption{Surface~$\Sigma_g$ and curves $\delta_1,\ldots,\delta_g,\epsilon_2,\ldots,\epsilon_{g-2}$}\label{figure_delta_epsilon}
\end{figure}

Let $\CV$ be the set of all \textit{unordered} splittings
$$
H_1(\Sigma_g)= W_1\oplus\cdots\oplus W_g
$$ 
such that $\rank W_i=2$, $i=1,\ldots,g$, and $W_i$ and $W_j$ are orthogonal with respect to the intersection form unless $i=j$. The word `unordered' means that we do not distinguish between splittings that are obtained from each other by permutations of summands. Once we have realized the surface~$\Sigma_g$ as shown in Fig.~\ref{figure_delta_epsilon}, we obtain a distinguished splitting $\W_{0}=\{W_{0,1},\ldots,W_{0,g}\}\in \CV$ consisting of the subgroups 
$$W_{0,i}=H_1(\bT_i)\subset H_1(\Sigma_g),$$ 
where $\bT_i\subset \Sigma_g$ is the  one-punctured torus bounded by~$\delta_i$.

The symplectic group~$\Sp(2g,\Z)$ acts transitively on~$\CV$. The stabilizer of the (unordered) splitting~$\W_0$ is the semi-direct product
$$
\CH=\left(\underbrace{\SL(2,\Z)\times\cdots\times\SL(2,\Z)}_{g}\right)\rtimes S_g,
$$
where every factor~$\SL(2,\Z)$ acts by special linear transformations of the corresponding summand~$W_{0,i}\cong\Z^2$ and the symmetric group~$S_g$ permutes the summands. Obviously,   $\CH$ has infinite index in~$\Sp(2g,\Z)$.

\begin{theorem}\label{theorem_Abelian}
Suppose that $g\ge 3$. Let $X_0=1,X_1,X_2,\ldots\in\Sp(2g,\Z)$ be representatives of all left cosets of~$\CH$ in~$\Sp(2g,\Z)$ and let $\varphi_0=1,\varphi_1,\varphi_2,\ldots\in\Mod_g$ be mapping classes that go to $X_0,X_1,X_2,\ldots$, respectively, under the natural surjective homomorphism $\Mod_g\twoheadrightarrow \Sp(2g,\Z)$. Then the abelian cycles 
$$
\iA_s=(\varphi_{s})_*(\iA_0)=\iA\left(T_{\varphi_s(\delta_1)},\ldots,T_{\varphi_s(\delta_g)},T_{\varphi_s(\epsilon_2)},\ldots,T_{\varphi_s(\epsilon_{g-2})}\right),\quad s=0,1,2,\ldots
$$ 
form a basis of a free $\Z[U]$-submodule of~$H_{2g-3}(\K_g)$. In other words, the homology classes 
$$
t^{\theta}\iA_s,\qquad \theta\in U,\ s=0,1,2,\ldots
$$
are linearly independent, i.\,e., form a basis of a free abelian subgroup of~$H_{2g-3}(\K_g)$.
\end{theorem}

This theorem provides the first explicit construction of infinitely many linearly independent homology classes in $H_*(\K_g,\Q)$, thus giving solution to Problem~5.15 in~\cite{Far06}.

Obviously, Theorem~\ref{theorem_main} follows from Theorem~\ref{theorem_Abelian}. 
Let us show that Theorem~\ref{theorem_main2} also follows from Theorem~\ref{theorem_Abelian}.

\begin{proof}[Proof of Theorem~\ref{theorem_main2}] 
Since $\Q$ is a flat $\Z$-module, one can easily deduce from Theorem~\ref{theorem_Abelian} that the images of the elements~$\iA_0,\iA_1,\iA_2,\ldots$ in~$H_{2g-3}(\K_g,\Q)$ form a basis of a free $\Q[U]$-submodule of it. Since $U$ is a free abelian group of rank $r={2g\choose 3}-2g$, we obtain that $\Q[U]$ is isomorphic to the ring of Laurent polynomials in $r$ independent variables with coefficients in~$\Q$. Hence $\Q[U]$ is a Noetherian commutative ring. To complete the proof of Theorem~\ref{theorem_main2} it is enough to notice that a finitely generated module over a Noetherian ring cannot contain an infinitely generated free submodule.
\end{proof}

Two main ingredients of our proof of Theorem~\ref{theorem_Abelian}  are the Cartan--Leray spectral sequence for the action of~$\K_g$ on the complex of cycles~$\B_g$ constructed by  Bestvina, Bux, and Margalit (see~\cite{BBM07}), and Morita's construction of homomorphisms $\K_g\to\Z$ related to the Casson invariant of homology $3$-spheres (see~\cite{Mor89}, \cite{Mor91}). 

Either of the groups~$H_{2g-3}(\K_g)$ and~$H_{2g-3}(\I_g)$ contains a free abelian subgroup of infinite rank generated by abelian cycles. (For~$\I_g$ this was proved by the author in~\cite{Gai18}.) However, the abelian cycles in~$H_{2g-3}(\I_g)$ that were proved in~\cite{Gai18} to generate a free abelian subgroup of infinite rank do not belong to the image of the homomorphism $H_{2g-3}(\K_g)\to H_{2g-3}(\I_g)$ induced by the inclusion. The following question naturally arises. 

\begin{quest}
Is the image of the homomorphism $H_{2g-3}(\K_g)\to H_{2g-3}(\I_g)$ infinitely generated? Is the image of the homomorphism $H_{2g-3}(\K_g,\Q)\to H_{2g-3}(\I_g,\Q)$ infinite-dimensional?
\end{quest}

\section{Casson invariant and cohomology class $\Phi$}

The aim of this section is to prove the following proposition.

\begin{propos}\label{propos_infinite}
The homology class~$\iA_0\in H_{2g-3}(\K_g)$ given by~\eqref{eq_A} is nonzero and has an infinite order.
\end{propos}

To prove this, we shall construct a cohomology class $\Phi\in H^{2g-3}(\K_g)$ satisfying $\langle\Phi,\iA_0\rangle\ne 0$. The class~$\Phi$ will be the product of $2g-3$ one-dimensional cohomology classes corresponding to certain Morita homomorphisms $\K_g\to\Z$. This construction developes the approach due to Brendle and Farb~\cite{BrFa07} who used Morita homomorphisms to prove the linear independence of certain abelian cycles in $H^2(\K_g)$.

By the fundamental theorem of Casson, there exists a well-defined integer valued invariant~$\lambda$ for oriented homology $3$-spheres that satisfies the following conditions (cf.~\cite{Sav99}):
\begin{enumerate}
\item The reduction of~$\lambda$ modulo~$2$ is the Rokhlin invariant.
\item $\lambda(-M)=-\lambda(M)$.
\item $\lambda(M_1\# M_2)=\lambda(M_1)+\lambda(M_2)$.
\item Suppose that $k$ is a knot in an oriented homology $3$-sphere~$M$ and $M_n(k)$ is the homology $3$-sphere obtained from~$M$ by performing the $(1/n)$-surgery on~$k$. Then 
\begin{equation*}
\lambda\bigl(M_n(k)\bigr)=\lambda(M)+\frac{n}{2}\,\Delta''_k(1),
\end{equation*}
where $\Delta_k(t)$ is the Alexander polynomial of~$k$ normalized so that $\Delta_k(t^{-1})=\Delta_k(t)$ and $\Delta_k(1)=1$.
\end{enumerate}

Recall Morita's construction of homomorphisms $\K_g\to\Z$ based on the Casson invariant, see~\cite{Mor89},~\cite{Mor91}.

Let  $M$ be an oriented  homology $3$-sphere and let $f\colon \Sigma_g\hookrightarrow M$ be a Heegaard embedding. Then the surface $f(\Sigma_g)$ splits $M$ into two handlebodies, which we denote by $V_+$ and $V_-$ so that the orientation of~$V_+$ induces the orientation of~$\partial V_+=f(\Sigma_g)$ coinciding with those obtained by applying~$f$ to the given orientation of~$\Sigma_g$. Let 
$$
l_f\colon H_1(\Sigma_g)\times H_1(\Sigma_g)\to\Z
$$
be the Seifert bilinear form given by
$$
l_f(x,y)=\lk\bigl(f_*(x),f_*(y)^+\bigr),
$$
where $f_*(y)^+$ is the cycle in $M\setminus f(\Sigma_g)$ obtained by pushing the cycle $f_*(y)$ to the positive direction, and $\lk(\cdot,\cdot)$ denotes the linking number. (Notice that the cycle~$f_*(y)^+$ lies in~$V_-$, since the positive normal to~$f(\Sigma_g)$ is the outer normal for~$V_+$.) It is easy to check that
$$
l_f(y,x)=l_f(x,y)+ x\cdot y.
$$

For each $h\in\Mod_g$, we denote by~$M_h$ the oriented $3$-manifold obtained by cutting~$M$ along~$f(\Sigma_g)$ and then regluing the handlebodies~$V_+$ and~$V_-$ by the following homeomorphism of their boundaries:
$$
\partial V_+=f(\Sigma_g)\xrightarrow{f^{-1}}\Sigma_g\xrightarrow{h}\Sigma_g\xrightarrow{f}f(\Sigma_g)=\partial V_-.
$$
If $h$ belongs to the Torelli group~$\I_g$, then $M_h$ is again a homology sphere.
Consider the mapping $\lambda_f\colon \I_g\to\Z$ given by
\begin{equation*}
\lambda_f(h)=\lambda(M_h)-\lambda(M).
\end{equation*}

\begin{theorem}[{Morita~\cite[Theorem~2.2]{Mor91}}]
The restriction of~$\lambda_f$ to~$\K_g$ is a homomorphism. If $\gamma$ is a separating simple closed curve in~$\Sigma_g$ and $\Sigma'$ is one of the two connected components of $\Sigma_g\setminus \gamma$, then
\begin{equation}\label{eq_lambda}
\begin{split}
\lambda_f(T_{\gamma})
=&\sum_{i=1}^{g'}\sum_{j=1}^{g'}\bigl(l_f(a_i,a_j)l_f(b_i,b_j)-l_f(a_i,b_j)l_f(b_i,a_j)\bigr)\\
=&\sum_{i=1}^{g'}\bigl(l_f(a_i,a_i)l_f(b_i,b_i)-l_f(a_i,b_i)l_f(b_i,a_i)\bigr)\\
&+2\sum_{1\le i<j\le g'}\bigl(l_f(a_i,a_j)l_f(b_i,b_j)-l_f(a_i,b_j)l_f(a_j,b_i)\bigr),
\end{split}
\end{equation}
where $a_1,b_1,\ldots,a_{g'},b_{g'}$ is a symplectic basis of the subgroup $H_1(\Sigma')\subseteq H_1(\Sigma_g)$.
\end{theorem}

\begin{remark}
The sign in formula~\eqref{eq_lambda} is opposite to the sign in the corresponding formula in~\cite{Mor91}, since Morita used right Dehn twists and we are using left Dehn twists.
\end{remark}

Now, we choose several special Heegaard embeddings~$f$ of~$\Sigma_g$ into the standard sphere~$S^3$ and study the corresponding homomorphisms~$\lambda_f$.
Let $f_0\colon \Sigma_g\hookrightarrow S^3$ be the Heegaard embedding shown in Fig.~\ref{figure_delta_epsilon}. (We identify~$S^3$ with the one-point compactification of~$\R^3$.) In this figure, $V_+$ is the handlebody inside~$\Sigma_g$ and $V_-$ is the handlebody outside~$\Sigma_g$. For every $i=1,\ldots,g$, we choose a symplectic basis $A_i, B_i$ of~$W_{0,i}=H_1(\bT_i)$ such that a curve of homology class~$A_i$ bounds a disk in~$V_+$ and a curve of homology class~$B_i$ bounds a disk in~$V_-$, see Fig.~\ref{figure_delta_epsilon}. Then $A_1,B_1,\ldots,A_{g},B_{g}$ is a symplectic basis of~$H_1(\Sigma_g)$. It is easy to see that 
\begin{gather*}
l_{f_0}(A_i,A_j)=l_{f_0}(B_i,B_j)=l_{f_0}(A_i,B_j)=0,\\
l_{f_0}(B_i,A_j)=
\left\{
\begin{aligned}
&1&&\text{if $i=j$,}\\
&0&&\text{if $i\ne j$.}
\end{aligned}
\right.
\end{gather*}

Now, suppose that $1\le p<q\le g$. Choose an arbitrary mapping class $\psi_{p,q}\in\Mod_g$ that acts on~$H_1(\Sigma_g)$ by the following symplectic transformation:
\begin{align*}
A_p&\mapsto A_p+A_q+B_q,&
B_p&\mapsto A_p+B_p,\\ 
A_q&\mapsto A_q+A_p+B_p,&
B_q&\mapsto A_q+B_q,\\
A_i&\mapsto A_i,&  B_i&\mapsto B_i,\qquad\qquad i\ne p,q.
\end{align*}
Consider the Heegaard embedding 
$$
f_{p,q} = f_0\circ \psi_{p,q}\colon\Sigma_g\hookrightarrow S^3.
$$ 
Then 
$$
l_{f_{p,q}}(x,y)=l_{f_0}\bigl(\psi_{p,q}(x),\psi_{p,q}(y)\bigr).
$$
It is easy to compute that 
\begin{align*}
l_{f_{p,q}}(A_p,A_p)=l_{f_{p,q}}(B_p,A_p)=l_{f_{p,q}}(B_p,B_p)&=1,&l_{f_{p,q}}(A_p,B_p)&=0,\\
l_{f_{p,q}}(A_q,A_q)=l_{f_{p,q}}(B_q,A_q)=l_{f_{p,q}}(B_q,B_q)&=1,&l_{f_{p,q}}(A_q,B_q)&=0,\\
l_{f_{p,q}}(A_p,A_q)=l_{f_{p,q}}(A_p,B_q)=l_{f_{p,q}}(B_p,A_q)&=1,&l_{f_{p,q}}(B_p,B_q)&=0.\\
\end{align*}
Besides, 
\begin{gather*}
l_{f_{p,q}}(A_i,A_j)=l_{f_{p,q}}(B_i,B_j)=l_{f_{p,q}}(A_i,B_j)=0,\\
l_{f_{p,q}}(B_i,A_j)=
\left\{
\begin{aligned}
&1&&\text{if $i=j$,}\\
&0&&\text{if $i\ne j$,}
\end{aligned}
\right.
\end{gather*}
provided that at least one of the two indices~$i$ and~$j$ is neither~$p$ nor~$q$.

We put $\lambda_{p,q}=\lambda_{f_{p,q}}$. Then, using~\eqref{eq_lambda}, we immediately obtain:
\begin{align}\label{eq_lambda_delta}
\lambda_{p,q}(T_{\delta_i})&=
\left\{
\begin{aligned}
&1&&\text{if $i=p,q$,}\\
&0&&\text{if $i\ne p,q$,}
\end{aligned}
\right.\\
\label{eq_lambda_eps}
\lambda_{p,q}(T_{\epsilon_i})&=
\left\{
\begin{aligned}
&1&&\text{if $p\le i<q$,}\\
&0&&\text{if $i<p$ or $i\ge q$.}
\end{aligned}
\right.
\end{align}

The homomorphisms~$\lambda_{p,q}\colon\K_g\to\Z$ are cohomology classes in $H^1(\K_g)$. We put 
$$
\Phi=(\lambda_{1,2}\lambda_{1,3}\cdots\lambda_{1,g})(\lambda_{2,3}\lambda_{3,4}\cdots\lambda_{g-1,g})\in H^{2g-3}(\K_g).
$$ 
For multiplication in cohomology we use the sign convention as in~\cite[Chapter~V]{Bro82}.

\begin{propos}\label{propos_Phi_infinite}
$\langle\Phi,\iA_0\rangle=(-1)^{g-1}\cdot2^{g-2}$.
\end{propos}

\begin{proof}
Put $\kappa_i=\lambda_{1,i+1}$ for $i=1,\ldots,g-1$ and $\kappa_i=\lambda_{i-g+2,i-g+3}$ for $i=g,\ldots,2g-3$. Put $h_i=T_{\delta_i}$ for $i=1,\ldots,g$ and $h_i=T_{\epsilon_{i-g+1}}$ for $i=g+1,\ldots,2g-3$. Then
$$
\langle\Phi,\iA_0\rangle=\langle\kappa_1\cdots\kappa_{2g-3},\iA(h_1,\ldots,h_{2g-3})\rangle=(-1)^{2g-3 \choose 2}\det C=(-1)^{g}\det C,
$$
where $C=(\kappa_i(h_j))$. 

Using~\eqref{eq_lambda_delta} and~\eqref{eq_lambda_eps}, we see that

\begin{equation*}
C=
\begin{pmatrix}
1 & 1 & 0 & 0 & 0 & \hdots & 0 & 0 & \vline & 0 & 0 & 0 & \hdots & 0\\
1 & 0 & 1 & 0 & 0 & \hdots & 0 & 0 & \vline & 1 & 0 & 0 & \hdots & 0\\
1 & 0 & 0 & 1 & 0 & \hdots & 0 & 0 & \vline & 1 & 1 & 0 & \hdots & 0\\
1 & 0 & 0 & 0 & 1 & \hdots & 0 & 0 & \vline & 1 & 1 & 1 & \hdots & 0\\
\vdots & \vdots &\vdots &\vdots &\vdots & \ddots &\vdots &\vdots & \vline & \vdots & \vdots  &\vdots & \ddots & \vdots \\
1 & 0 & 0 & 0 & 0 & \hdots & 1 & 0 & \vline & 1 & 1 & 1 & \hdots & 1\\
1 & 0 & 0 & 0 & 0 & \hdots & 0 & 1 & \vline & 1 & 1 & 1 & \hdots & 1\\
\hline
0 & 1 & 1 & 0 & 0 & \hdots & 0 & 0 & \vline & 1 & 0 & 0 & \hdots & 0\\
0 & 0 & 1 & 1 & 0 &  \hdots & 0 & 0 & \vline &0 & 1 & 0 & \hdots & 0\\
0 & 0 & 0 & 1 & 1 &\hdots & 0 & 0 & \vline &0 & 0 & 1 & \hdots & 0\\
\vdots & \vdots &\vdots &\vdots &\vdots  & \ddots & \vdots & \vdots & \vline &\vdots & \vdots  &\vdots & \ddots & \vdots\\
0 & 0 & 0 & 0 & 0 & \hdots & 1 &0 & \vline &0 & 0 & 0 & \hdots & 1\\
0 & 0 & 0 & 0 & 0 & \hdots & 1 &1 & \vline &0 & 0 & 0 & \hdots & 0
\end{pmatrix}
\end{equation*}
where the top left, top right, bottom left, and bottom right blocks have sizes $(g-1)\times g$, $(g-1)\times (g-3)$, $(g-2)\times g$, and $(g-2)\times (g-3)$, respectively. 
Now, it is not hard to compute that $\det C = -2^{g-2}$.
\end{proof}

Proposition~\ref{propos_infinite} follows immediately from Proposition~\ref{propos_Phi_infinite}.

\section{Complex of cycles and Cartan--Leray spectral sequence}\label{section_background}

Recall in a convenient for us form some definitions and results of~\cite{BBM07}. 
Denote by~$\cur$ the set of isotopy classes of oriented non-separating simple closed curves on~$\Sigma_g$. A finite union of pairwise disjoint oriented simple closed curves on~$\Sigma_g$ is said to be an \textit{oriented multicurve} if none of these curves is null-homotopic and no pair of these curves are homotopic to each other  with either preserving or reversing the orientation. All curves and multicurves will be always considered up to isotopy, and for the sake of simplicity, we shall later not distinguish between a curve or a multicurve and its isotopy class.

A \textit{basic $1$-cycle} for a homology class~$x\in H_1(\Sigma_g)$ is  a finite formal linear combination $\boldsymbol{\gamma}=\sum_{i=1}^k n_i\gamma_i$, where $n_i$ are positive integers and $\gamma_i\in\cur$, satisfying the following conditions:
\begin{enumerate}
\item The curves $\gamma_1,\ldots,\gamma_k$ are pairwise disjoint. (More precisely, the isotopy classes~$\gamma_1,\ldots,\gamma_k$ contain representatives that are pairwise disjoint.)
\item The homology classes $[\gamma_1],\ldots,[\gamma_k]$ are linearly independent.
\item $\sum_{i=1}^k n_i[\gamma_i]=x$.
\end{enumerate} 
Then the union $\gamma_1\cup\cdots\cup\gamma_k$ is an oriented multicurve without separating components. This multicurve is called the \textit{support} of~$\boldsymbol{\gamma}$. 

Choose a nonzero homology class $x\in H_1(\Sigma_g)$. 
Denote by~$\M(x)$ the set of isotopy classes of oriented multicurves~$M$ on~$\Sigma_g$ satisfying the following conditions: 
\begin{enumerate}
\item For each component~$\gamma$ of~$M$, there exists a basic $1$-cycle for~$x$ whose support  is contained in~$M$ and contains~$\gamma$.

\item Any nontrivial linear combination of the homology classes of components of~$M$ with nonnegative coefficients is nonzero.
\end{enumerate}

For an oriented multicurve~$M\in\M(x)$, let~$P_M$ be the convex hull in~$(\R_+)^{\cur}$ of all basic $1$-cycles for~$x$ with supports contained in~$M$. Since there are finitely many such basic $1$-cycles, $P_M$ is a finite-dimensional convex polytope.  It is not hard to check that every face of~$P_M$ is $P_N$ for an oriented  multicurve $N\in\M(x)$ contained in~$M$, and the intersection of any two polytopes~$P_{M_1}$ and~$P_{M_2}$ is either empty or a face~$P_N$ of both of them. The union of the polytopes~$P_M$ for all $M\in\M(x)$ is the regular cell complex~$\B_g(x)$, which was introduced by Bestvina, Bux, and Margalit~\cite{BBM07}  and was called by them the \textit{complex of cycles}.

We denote by~$\M_0(x)$ the subset of~$\M(x)$ consisting of all~$M$ such  that $\dim P_M=0$, i.\,e., $P_M$ is a vertex of~$\B_g(x)$. Obviously, $M\in\M_0(x)$ if and only if $M$ is the support of a basic $1$-cycle for~$x$.

\begin{theorem}[Bestvina, Bux, Margalit~\cite{BBM07}]\label{theorem_BBM} 
Suppose that $g\ge 2$ and $x\in H_1(\Sigma_g)$ is a nonzero homology class.
Then the complex of cycles~$\B_g(x)$ is contractible.
\end{theorem}

The group~$\I_g$ and its subgroup~$\K_g$ act cellularly and without rotations on~$\B_g(x)$. `Without rotations' means that if an element $h\in\I_g$ stabilizes a cell~$P_M$, then it stabilizes every point of~$P_M$.

Recall some standard facts on the Cartan--Leray spectral sequence, see~\cite{Bro82} for details. Suppose that a (discrete) group~$G$ acts cellularly and without rotations on a cell complex~$X$. Then there exists the Cartan--Leray spectral sequence~$E^*_{*,*}$ satisfying the following conditions:
\begin{enumerate}
\item The page~$E^1$ is given by
\begin{equation}\label{eq_E1pq}
E^1_{p,q}\cong\bigoplus_{\sigma\in \mathcal{X}_p}H_q(\Stab_{G}(\sigma)),
\end{equation}
where $\mathcal{X}_p$ is any set containing exactly one representative in every $G$-orbit of $p$-cells of~$X$, and $\Stab_{G}(\sigma)$ denotes the stabilizer of~$\sigma$.
\item The differential $d^r$ has bidegree $(-r,r-1)$, $r=1,2,\ldots$.
\item  $\bigoplus_{p+q=s}E^{\infty}_{p,q}$ is the adjoint graded group for certain filtration in the equivariant homology group~$H_s^G(X)$. In particular, there is a natural embedding $E^{\infty}_{0,s}\hookrightarrow H_s^G(X)$. Moreover, the isomorphism~\eqref{eq_E1pq} turns the composite homomorphism 
$$E^1_{0,s}\twoheadrightarrow E^{\infty}_{0,s}\hookrightarrow H_s^G(X)$$ 
into the sum of the homomorphisms 
$$
H_s(\Stab_{G}(v))=H_s^{\Stab_{G}(v)}(v)\to H_s^G(X),\qquad v\in\mathcal{X}_0,
$$
induced by the inclusions $v\subseteq X$ and $\Stab_{G}(v)\subseteq G$.
\end{enumerate}

If $X$ is contractible, then $H_*^G(X)\cong H_*(G)$.

Now, let $E^*_{*,*}$ be the Cartan--Leray spectral sequence for the action of~$\K_g$ on the contractible complex of cycles~$\B_g(x)$. Obviously, $\Stab_{\K_g}(P_M)=\Stab_{\K_g}(M)$.
Bestvina, Bux, and Margalit~\cite[Proposition~6.2]{BBM07} proved that  
$$
\cd\bigl(\Stab_{\K_g}(M)\bigr)\le 2g-3-\dim P_M
$$
for all $M\in\M(x)$. This immediately implies that $E^1_{p,q}=0$ whenever $p+q>2g-3$. Therefore, all differentials~$d^r$ to the group~$E^1_{0,2g-3}$ are trivial, hence, $E^{\infty}_{0,2g-3}=E^1_{0,2g-3}$. Thus, the above property~(3) of Cartan--Leray spectral sequence yields the following.

\begin{propos}\label{propos_embed}
Let $\mathfrak{M}\subseteq\M_0(x)$ be a subset consisting of oriented multicurves in pairwise different $\K_g$-orbits. Then the inclusions $\Stab_{\K_g}(M)\subseteq\K_g$, $M\in\mathfrak{M}$, induce the injective homomorphism
$$
\bigoplus_{M\in \mathfrak{M}}H_{2g-3}\bigl(\Stab_{\K_g}(M)\bigr)\hookrightarrow H_{2g-3}(\K_g).
$$
\end{propos}

\section{Proof of Theorem~\ref{theorem_Abelian}}

Recall that a nonzero element $a\in H_1(\Sigma_g)$ is called \textit{primitive} if $a$ is divisible by no integer greater than~$1$. For a primitive $a\in H_1(\Sigma_g)$, let $\T_a\colon H_1(\Sigma_g)\to H_1(\Sigma_g)$ be the transformation given by $\T_a(x)=x+(x\cdot a)a$. Then for any simple closed curve in homology class~$a$, the Dehn twist about it acts on~$H_1(\Sigma_g)$ by~$\T_a$.

To a symplectic basis $a_1,b_1,\ldots, a_g,b_g$ of~$H_1(\Sigma_g)$, we assign the splitting $U=U_a\oplus U_b$ such that $U_a$ is generated by all elements $a_i\wedge a_j\wedge a_k$ and $a_i\wedge a_j\wedge b_k$  and $U_b$ is generated by all elements  $a_i\wedge b_j\wedge b_k$ and $b_i\wedge b_j\wedge b_k$. 

The following lemma is proved by a straightforward computation.

\begin{lem}\label{lem_invariant}
Let $a_1,b_1,\ldots, a_g,b_g$ be a symplectic basis of~$H_1(\Sigma_g)$. Then an element $\theta\in U$ satisfies $\T_{a_i}(\theta)=\theta$ for all $i=1,\ldots,g$ if and only if  $\theta$ belongs to the subgroup of~$U$ generated by the~${g\choose 3}$ elements $a_i\wedge a_j\wedge a_k$, $i<j<k$, and the $g(g-1)$ elements $a_i\wedge a_j\wedge b_j$, $i\ne j$. In particular, any such~$\theta$ lies in~$U_a$.
\end{lem}

\begin{propos}\label{propos_indep1}
The homology classes $t^{\theta}\iA_0$, where $\theta$ runs over~$U$, are linearly independent in~$H_{2g-3}(\K_g)$.
\end{propos}

\begin{proof}
Assume the converse and consider a nontrivial linear relation 
\begin{equation}\label{eq_linear_relation}
\sum_{q=1}^Q m_qt^{\theta_q}\iA_0=0,\qquad m_q\in\Z,\ \theta_q\in U,
\end{equation}
with the smallest possible number~$Q$ of summands. Then $\theta_1,\ldots,\theta_Q$ are pairwise different and $m_1,\ldots,m_Q$ are nonzero.

By Proposition~\ref{propos_infinite}, $\iA_0$ has infinite order. Hence $t^{\theta}\iA_0$ have infinite order for all $\theta\in U$. Therefore $Q>1$.

Let us show that there exists a symplectic basis $a_1,b_1,\ldots, a_g,b_g$ such that
\begin{enumerate}
\item for each~$i$, $a_i,b_i$ is a symplectic basis of~$W_{0,i}=H_1(\bT_i)$,
\item $\theta_2-\theta_1\notin U_a$.
\end{enumerate}
Indeed,  take any symplectic basis satisfying condition~(1). Assume that $\theta_2-\theta_1\in U_a$. Then $\theta_2-\theta_1\notin U_b$, since $\theta_1\ne \theta_2$. Replace every~$a_i$ with~$b_i$ and simultaneously replace every~$b_i$ with $-a_i$. The summands~$U_a$ and~$U_b$ will interchange. Hence we obtain the basis satisfying both required conditions.

We put $x=a_1+\cdots+a_g$ and consider the set of oriented multicurves~$\M_0(x)$. Every homology class~$a_i$ can be represented by an oriented simple closed curve~$\alpha_i$ lying in~$\bT_i$. Then the oriented multicurve $M=\alpha_1\cup\cdots\cup\alpha_g$ belongs to~$\M_0(x)$, since $\alpha_1+\cdots+\alpha_g$ is a basic $1$-cycle for~$x$. Since the curves $\delta_1,\ldots,\delta_g,\epsilon_2,\ldots,\epsilon_{g-2}$ are disjoint from~$M$, we obtain that $\iA_0$ lies in the image of the homomorphism $H_{2g-3}\bigl(\Stab_{\K_g}(M)\bigr)\to H_{2g-3}(\K_g)$ induced by the inclusion.

We put $V=\tau\bigl(\Stab_{\I_g}(M)\bigr)\subseteq U$. The Dehn twists~$T_{\alpha_1},\ldots,T_{\alpha_g}$ commute with all elements of~$\Stab_{\I_g}(M)$. Hence  $\T_{a_1},\ldots,\T_{a_g}$ stabilize every element of~$V$. By Lemma~\ref{lem_invariant}, $V\subseteq U_a$. Therefore $\theta_2-\theta_1\notin V$. Decompose the set of indices $\{1,\ldots,Q\}$ into pairwise disjoint subsets~$I_1,\ldots,I_R$ such that $q_1$ and~$q_2$ belong to the same subset~$I_r$ if and only if $\theta_{q_1}$ and $\theta_{q_2}$ belong to the same coset of~$V$ in~$U$. Since $1$ and~$2$ lie in different subsets~$I_{r_1}$ and~$I_{r_2}$, we obtain that $R>1$. Hence the cardinality of every~$I_r$ is strictly smaller than~$Q$.

For every $r=1,\ldots,R$, choose a mapping class $h_r\in\I_g$ such that $\tau(r)$ lies in the coset of~$V$ that contains the elements~$\theta_q$ for all $q\in I_r$. Then $\tau(h_1),\ldots,\tau(h_R)$ lie in pairwise different cosets of~$V$. Obviously $h_r(M)\in\M_0(x)$ for all~$r$. Notice that the multicurves~$h_1(M),\ldots,h_R(M)$ lie in pairwise different $\K_g$-orbits. Indeed, if there existed an element $k\in\K_g$ such that $kh_{r_1}(M)=h_{r_2}(M)$, $r_1\ne r_2$, then we would obtain that $h^{-1}_{r_2}kh_{r_1}\in\Stab_{\I_g}(M)$, hence, the element
$$
\tau\left(h^{-1}_{r_2}kh_{r_1}\right)=\tau(h_{r_1})-\tau(h_{r_2})
$$
would belong to~$V$, which is not true.

Now, it follows from Proposition~\ref{propos_embed} that the inclusions $\Stab_{\K_g}(h_r(M))\subseteq\K_g$ induce the injective homomorphism
$$
\iota\colon \bigoplus_{r=1}^RH_{2g-3}\bigl(\Stab_{\K_g}(h_r(M))\bigr)\hookrightarrow H_{2g-3}(\K_g).
$$
If $q\in I_r$, then we have $\theta_q-\tau(h_r)\in V$. Hence there exists an element $f_q\in \Stab_{\I_g}(M)$ such that $\tau(f_q)=\theta_q-\tau(h_r)$, i.\,e., $\tau(h_rf_q)=\theta_q$. Then $t^{\theta_q}\iA_0=(h_rf_q)_*(\iA_0)$. Since $\iA_0$ lies in the image of the homomorphism $H_{2g-3}\bigl(\Stab_{\K_g}(M)\bigr)\to H_{2g-3}(\K_g)$ induced by the inclusion and $h_rf_q(M)=h_r(M)$, we obtain that $$t^{\theta_q}\iA_0\in \iota\left(H_{2g-3}\bigl(\Stab_{\K_g}(h_r(M))\bigr)\right).$$
Now, it follows from~\eqref{eq_linear_relation} that
$$
\sum_{q\in I_r}m_qt^{\theta_q}\iA_0=0,\qquad r=1,\ldots,R,
$$
which contradicts the assumption of the minimality of~$Q$.
\end{proof}

\begin{cor}\label{cor_indep1}
For each~$s$, the homology classes $t^{\theta}\iA_s$, where $\theta$ runs over~$U$, are linearly independent in~$H_{2g-3}(\K_g)$.
\end{cor}

\begin{proof}
This follows  from Proposition~\ref{propos_indep1} and the formula $X_s(t^{\theta}\iA_0)=t^{X_s(\theta)}\iA_s$.
\end{proof}

\begin{propos}\label{propos_indep2}
Suppose that 
\begin{equation*}
\sum_{s=0}^Nc_s\iA_s=0,\qquad c_s\in \Z[U].
\end{equation*}
Then $c_s\iA_s=0$ for  all~$s$.
\end{propos}
\begin{proof}
Recall that $\W_{0}=\{W_{0,1},\ldots,W_{0,g}\}$ is the splitting of~$H_1(\Sigma_g)$ consisting of the subgroups $W_{0,i}=H_1(\bT_i)$. 
Consider the splittings 
$$
\W_{s}=\{W_{s,1},\ldots,W_{s,g}\},\qquad W_{s,i}=(\varphi_s)_*(W_{0,i}).
$$
Since the transformations $X_0,\ldots,X_N\in\Sp(2g,\Z)$ lie in pairwise different left cosets of~$\CH$, we see that the splittings $\W_0,\ldots,\W_N$ are pairwise different.

For a homology class $x\in H_1(\Sigma_g)$, we have unique decompositions
$$
x=x_{s,1}+\cdots+x_{s,g},\qquad x_{s,i}\in W_{s,i},\ s=0,\ldots,N.
$$

\begin{lem}\label{lem_generic}
There exists a homology class $x\in H_1(\Sigma_g)$ such that
\begin{enumerate}
\item all homology classes $x_{s,i}$ are nonzero, $s=0,\ldots,N$, $i=1,\ldots,g$,
\item if $0\le s<q\le N$, then $\{x_{s,1},\ldots,x_{s,g}\}\ne \{x_{q,1},\ldots,x_{q,g}\}$.
\end{enumerate}
\textnormal{(}Here $\{x_{s,1},\ldots,x_{s,g}\}$ is the unordered set.\textnormal{)}
\end{lem}

\begin{proof}
Let $\Pi_{s,i}\colon H_1(\Sigma_g)\to H_1(\Sigma_g)$ be the projection along $\bigoplus_{j\ne i}W_{s,j}$ onto~$W_{s,i}$. Then $x_{s,i}=\Pi_{s,i}(x)$. For each pair $(s,q)$ such that $0\le s<q\le N$, the splittings $\W_s$ and~$\W_q$ are different. Hence there exists a number~$i_{s,q}$ such that $W_{s,i_{s,q}}\ne W_{q,j}$ for all $j$. Then $\Pi_{s,i_{s,q}}\ne \Pi_{q,j}$  for all $j$.

Consider the $(N+1)g$ homomorphisms $\Pi_{s,i}$, $s=0,\ldots,N$, $i=1,\ldots,g$, and the $N(N+1)g/2$ homomorphisms $\Pi_{s,i_{s,q}}-\Pi_{q,j}$, $0\le s<q\le N$, $j=1,\ldots,g$. Each of these  homomorphisms is nontrivial. Hence its kernel is a subgroup of~$H_1(\Sigma_g)$ of rank strictly smaller than~$2g$. Therefore the union~$K$ of the kernels of all $(N+1)(N+2)g/2$ homomorphisms~$\Pi_{s,i}$ and~$\Pi_{s,i_{s,q}}-\Pi_{q,j}$ does not coincide with the whole group~$H_1(\Sigma_g)$. We take for $x$ an arbitrary element of~$H_1(\Sigma_g)\setminus K$. Then $x_{s,i}\ne 0$ for all~$s$ and~$i$ and $x_{s,i_{s,q}}-x_{q,j}\ne 0$ for all $s<q$ and all~$j$. Therefore $\{x_{s,1},\ldots,x_{s,g}\}\ne \{x_{q,1},\ldots,x_{q,g}\}$ whenever $0\le s<q\le N$.
\end{proof}

Let us proceed with the proof of Proposition~\ref{propos_indep2}.
Choose a homology class $x\in H_1(\Sigma_g)$ satisfying conditions~(1) and~(2) in Lemma~\ref{lem_generic}. Then every~$x_{s,i}$ can be uniquely written in the form $x_{s,i}=n_{s,i}a_{s,i}$, where $n_{s,i}$ is a positive integer and $a_{s,i}$ is a primitive element of~$W_{s,i}$. (Warning: Generally, it is  by no means true that $a_{s,i}=(\varphi_s)_*(a_{0,i})$.)

Notice that $\{a_{s,1},\ldots,a_{s,g}\}\ne \{a_{q,1},\ldots,a_{q,g}\}$ whenever $0\le s<q\le N$. Indeed, if there existed a permutation~$\nu$ such that $a_{q,\nu(i)}=a_{s,i}$, then we would obtain that 
$$
\sum_{i=1}^g(n_{s,i}-n_{q,\nu(i)})a_{s,i}=0.
$$
But the elements $a_{s,1},\ldots,a_{s,g}$ are linearly independent, since they are nonzero and belong to pairwise different summands of the splitting~$\W_s$. Therefore we would obtain that $n_{q,\nu(i)}=n_{s,i}$, hence,  $x_{q,\nu(i)}=x_{s,i}$ for all~$i$, which would contradict condition~(2) in Lemma~\ref{lem_generic}.

For each $s$, choose a homeomorphism in the mapping class~$\varphi_s$ and denote it again by~$\varphi_s$. Then $W_{s,i}=H_1(\varphi_s(\bT_i))$, $i=1,\ldots,g$. 
Hence the primitive homology class~$a_{s,i}\in W_{s,i}$ can be represented by an oriented simple closed curve $\alpha_{s,i}$ lying in $\varphi_s(\bT_i)$. For each~$s$, the curves~$\alpha_{s,1},\ldots,\alpha_{s,g}$ are pairwise disjoint and $\sum_{i=1}^gn_{s,i}\alpha_{s,i}$ is a basic $1$-cycle for~$x$. Therefore the  oriented multicurve $M_s=\alpha_{s,1}\cup\cdots\cup\alpha_{s,g}$ belongs to~$\M_0(x)$. Since  $\{a_{s,1},\ldots,a_{s,g}\}\ne \{a_{q,1},\ldots,a_{q,g}\}$ unless $s=q$, we obtain that the multicurves $M_0,\ldots,M_N$ lie in pairwise different $\I_g$-orbits.

For each~$s$, let $h_{s,1},h_{s,2},\ldots$ be representatives of all double cosets~$\K_g\backslash\I_g/\Stab_{\I_g}(M_s)$. The Johnson homomorphism~$\tau$ identifies $\K_g\backslash\I_g/\Stab_{\I_g}(M_s)$ with $U/V_s$, where $V_s=\tau\bigl(\Stab_{\I_g}(M_s)\bigr)$. So it follows from Lemma~\ref{lem_invariant} that $\bigl|\K_g\backslash\I_g/\Stab_{\I_g}(M_s)\bigr|=\infty$. Then the oriented multicurves $h_{s,r}(M_s)$, $r=1,2,\ldots$, belong to~$\M_0(x)$ and lie in pairwise different $\K_g$-orbits but in the same $\I_g$ -orbit. Therefore all oriented multicurves $h_{s,r}(M_s)$, $s=0,\ldots,N$, $r=1,2,\ldots$, lie in pairwise different $\K_g$-orbits. By Proposition~\ref{propos_embed}, the inclusions $\Stab_{\K_g}(h_{s,r}(M_s))\subseteq \K_g$ induce the injective homomorphism
$$
\iota\colon\bigoplus_{s=0}^N\bigoplus_{r=1}^{\infty} H_{2g-3}\bigl(\Stab_{\K_g}(h_{s,r}(M_s))\bigr)\hookrightarrow H_{2g-3}(\K_g).
$$

\begin{lem}\label{lem_Qs}
For each~$s$, the subgroup
$$
\Lambda_s=\iota\left(\bigoplus_{r=1}^{\infty} H_{2g-3}\bigl(\Stab_{\K_g}(h_{s,r}(M_s))\bigr)\right)\subseteq H_{2g-3}(\K_g)
$$
is a $\Z[U]$-submodule.
\end{lem}

\begin{proof}
If $H_1$ and~$H_2$ are conjugate subgroups of a group~$G$, then the  images of the homomorphisms $H_*(H_1)\to H_*(G)$ and $H_*(H_2)\to H_*(G)$ induced by the inclusions coincide. For each element $h\in\I_g$, we have $h(M_s)=kh_{s,r}(M_s)$ for some~$r$ and some $k\in\K_g$. Hence the subgroup $\Stab_{\K_g}(h(M_s))$ is conjugate in~$\K_g$ to the subgroup $\Stab_{\K_g}(h_{s,r}(M_s))$. Therefore $\Lambda_s$ is the subgroup spanned by the images of all homomorphisms $H_{2g-3}\bigl(\Stab_{\K_g}(h(M_s))\bigr)\to H_{2g-3}(\K_g)$ induced by the inclusions, where $h$ runs over~$\I_g$. Consequently, $\Lambda_s$ is $\I_g$-invariant, i.\,e., is a $\Z[U]$-submodule.
\end{proof}

For $s=0,\ldots,N$, the multicurve~$M_s$ is contained in the union of the tori $\varphi_s(\bT_i)$, hence, is disjoint from the curves $\varphi_s(\delta_1),\ldots,\varphi_s(\delta_g),\varphi_s(\epsilon_2),\ldots,\varphi_s(\epsilon_{g-2})$. Therefore $\iA_s\in\iota\left(H_{2g-3}\bigl(\Stab_{\K_g}(M_s)\bigr)\right)\subseteq \Lambda_s$. By Lemma~\ref{lem_Qs}, we have $c_s\iA_s\in \Lambda_s$. Since $\iota$ is injective, we obtain that $\bigoplus_{s=0}^N \Lambda_s$ is the submodule of~$H_{2g-3}(\K_g)$. Therefore $c_s\iA_s=0$ for all~$s$, which completes the proof of Proposition~\ref{propos_indep2}.
\end{proof}

Theorem~\ref{theorem_Abelian} follows immediately from Proposition~\ref{propos_indep2} and Corollary~\ref{cor_indep1}.

\end{document}